\documentclass{amsart}
\usepackage{amsmath,amssymb,amsfonts,amsthm}
\usepackage{mathtools}
\usepackage{upgreek} 
\usepackage{paralist}
\usepackage{graphics}
\usepackage{epsfig}
\usepackage{graphicx}  
\usepackage{epstopdf}
\usepackage[colorlinks=true]{hyperref}
\hypersetup{urlcolor=blue, citecolor=blue}

\DeclareMathAlphabet{\mathpzc}{OT1}{pzcm}{mb}{it}

\newtheorem{thm}{Theorem}[section]
\newtheorem{cor}[thm]{Corollary}
\newtheorem{lem}[thm]{Lemma}
\newtheorem{pro}[thm]{Proposition}
\newtheorem{exa}{Example}
\theoremstyle{definition}
\newtheorem{defn}[thm]{Definition}
\newtheorem{rem}{Remark}

\newcommand{\lie}[2]{\ensuremath{[\![ #1,#2 ]\!]}}
\newcommand{\se}{\Upgamma (\Uplambda}

\newcommand{\dd}{\mathrm{d}}
\title[Lie algebroids and cohomology operators]
      {Lie algebroids generated by cohomology operators}

\author[D Garc\'ia-Beltr\'an, J A Vallejo and Yu. Vorobiev]{}


 \email{dennise.gb@fc.uaslp.mx}
 \email{jvallejo@fc.uaslp.mx}
 \email{yurimv@guaymas.uson.mx}

\thanks{The second author was partially supported by a CONACyT project CB-179115}

\begin{document}
\maketitle

% Enter the first author's name and address:
\centerline{\scshape Dennise Garc\'ia-Beltr\'an}
\medskip
{\footnotesize
% please put the address of the first author
 \centerline{Departamento de Matem\'aticas, Universidad de Sonora}
   \centerline{Blvd. Encinas y Rosales, Edificio 3K-1}
   \centerline{Hermosillo, Son 83000, M\'exico}
} % Do not forget to end the {\footnotesize by the sign }

\medskip

\centerline{\scshape Jos\'e A. Vallejo}
\medskip
{\footnotesize
 % please put the address of the second  and third author
 \centerline{Facultad de Ciencias, Universidad Aut\'onoma de San Luis Potos\'i}
   \centerline{Lat. Av. Salvador Nava s/n Col. Lomas}
   \centerline{San Luis Potos\'i, SLP 78290, M\'exico}
}

\medskip

\centerline{\scshape Yurii Vorobiev}
\medskip
{\footnotesize
% please put the address of the first author
 \centerline{Departamento de Matem\'aticas, Universidad de Sonora}
   \centerline{Blvd. Encinas y Rosales, Edificio 3K-1}
   \centerline{Hermosillo, Son 83000, M\'exico}
} 

\bigskip

% The name of the associate editor will be entered by an editorial staff
% "Communicated by the associate editor name" is not needed for special issue.
% \centerline{(Communicated by the associate editor name)}

%The abstract of your paper
\begin{abstract}
By studying the Fr\"olicher-Nijenhuis decomposition of cohomology operators (that is, 
derivations $D$ of the exterior algebra $\Omega (M)$ with $\mathbb{Z}-$degree $1$ and $D^2=0$),
we describe new examples of Lie algebroid structures on the tangent bundle $TM$ 
(and its complexification $T^{\mathbb{C}}M$)
constructed from pre-existing geometric ones such as complex, product or tangent structures.
We also describe a class of Lie algebroids on tangent bundles associated to idempotent endomorphisms with nontrivial Nijenhuis torsion. 
\end{abstract}

%The title of your section 1
\section{Introduction}

In this paper, we present an algebraic approach to the study of the relationship between Lie algebroids and cohomology operators which is based on the Fr\"olicher-Nijenhuis calculus (various approaches to this issue can be found, for example, in \cite{cramoe04, gra1301, kosmag90, mac05}).\\
Our idea comes from the relation between the exterior differential and the Lie bracket of vector fields on a manifold $M$.
As a first-order differential operator on the (sections of) exterior algebra, 
$\Omega (M)=\Upgamma \Uplambda T^*M$, the exterior differential can be characterized by giving
its action on generators:
$$
\begin{dcases}
\mathrm{d}f(X)=Xf, \\
\mathrm{d}\alpha (X,Y)=X\alpha (Y)-Y\alpha (X)-\alpha ([X,Y]),
\end{dcases}
$$
where $f\in\mathcal{C}^\infty (M),\alpha\in\Omega (M),X,Y\in\mathcal{X}(M)$ are arbitrary. 
The action is then extended to the whole of $\Omega (M)$ as a derivation of $\mathbb{Z}-$degree $1$.
Notice that the cohomology property $\mathrm{d}^2=0$ is equivalent to the Jacobi identity for $[\, ,\, ]$.\\
This setup allows us to reverse the definitions. Starting from the exterior differential $\mathrm{d}$
on $\Omega (M)$, we could define the Lie bracket on $\mathcal{X}(M)$ as follows: given $X,Y\in\mathcal{X}(M)$,
their bracket is the unique element $[X,Y]\in\mathcal{X}(M)$ such that, 
$$
\alpha ([X,Y])=X\alpha(Y)-Y\alpha (X)-\mathrm{d}\alpha (X,Y),
$$
\emph{for any }$\alpha\in\Omega^1(M)$.
In particular, if $\alpha=\mathrm{d}f$, for $f\in\mathcal{C}^\infty(M)$, 
the preceding formula along with $\mathrm{d}^2=0$ gives
$$
[X,Y]f=\mathrm{d}f([X,Y])=X\mathrm{d}f(Y)-Y\mathrm{d}f(X)= X(Yf)-Y(Xf),
$$
so the defined bracket certainly coincides with the Lie one. Again, 
the Jacobi identity is readily seen to be equivalent to the coboundary condition $\mathrm{d}^2=0$.

The correspondence between brackets and cohomology operators can be extended to the setting of Lie algebroids.
Recall that a Lie algebroid on a manifold $M$ consists of a vector bundle $E$ over $M$, together with a vector bundle map
$q:E\to TM$ over $M$ (called the anchor map), and a Lie bracket on sections $[\![\, ,\, ]\!]:\Upgamma E\times
\Upgamma E\to \Upgamma E$ satisfying the Leibniz rule
$$
[\![A,fB]\!]=f[\![A,B]\!]+qA(f)B,
$$
for all $A,B\in\Upgamma E$, $f\in\mathcal{C}^\infty (M)$. It can be thought as a replacement of the tangent bundle
$TM$, joint with the Lie bracket on vector fields, by the new bundle $E$ and the bracket 
$[\![\, ,\, ]\!]$ on its sections. For a
comprehensive reference on Lie algebroids, see \cite{mac05}.
The relevance of Lie algebroids in Mechanics is
explained in detail in works such as \cite{wei96,cle01,mar07} and references therein.

We first describe a class of Lie algebroids arising from an idempotent endomorphism with nontrivial Nijenhuis torsion, and then construct new examples of Lie algebroids. More
precisely, the paper is organized as follows: Section \ref{sec1} of this paper contains
a brief r\'esum\'e of the Fr\"olicher-Nijenhuis calculus 
for derivations of the exterior algebra over a vector bundle. In section \ref{sec2} we will recall 
how, given a Lie algebroid $E$ on $M$, one can construct a cohomology operator on 
$\Upgamma \Uplambda E^*$, and study its Fr\"olicher-Nijenhuis decomposition. It turns out 
that this decomposition can be described within the framework of Lie algebroid
connections \cite{fer02,mac05,blao06}.
In section \ref{sec3} we will follow the reverse path, constructing Lie algebroids on the tangent 
bundle $E=TM$ from a given cohomology operator $D$ on $\Omega (M)$. Our main tool here will be 
again the Fr\"olicher-Nijenhuis decomposition of derivations, which we apply in section \ref{sec4} 
to the case of idempotent endomorphisms with nonzero Nijenhuis decomposition of the tangent bundle. 
%In this case, we prove that the cohomology ring determined by $D$ is isomorphic to the foliated cohomology given by the associated Lie algebroid, without the need for a regularity condition on the underlying foliation.
%The algebroid obtained in this way is defined on the whole $TM$, not just on the tangent bundle
%of the foliation, as is done in the usual example of a Lie algebroid associated to a foliation.
In section \ref{sec5}, on the framework of our approach, we discuss Lie algebroids associated to regular foliations.
Finally, sections \ref{sec6} to \ref{sec8} are devoted to introducing
new examples of Lie algebroids canonically associated to a generalized foliation, and to
complex, product and tangent manifolds
(the latter, with the aid of a connection defined through a semispray).

\section{The Fr\"olicher-Nijenhuis decomposition}\label{sec1}
The proofs of the results stated in this section are slight generalizations of those corresponding
to the real case, which can be found, e.g., in \cite{mic87} or \cite{ksm94}.

Let $M$ be a differential manifold with tangent bundle $TM$. Its complexified bundle is then $T_{\mathbb{C}}M=TM\oplus iTM$. Complex vector fields are of the form $Z=X+iY$, with 
$X,Y\in\Upgamma TM$, and complex $1-$forms are constructed as the duals
$\Omega^1_{\mathbb{C}}(M)=\Upgamma (T_{\mathbb{C}}^*M)\simeq (\Upgamma T^*M)\oplus i(\Upgamma T^*M) 
\simeq\Omega^1(M)\oplus i\Omega^1(M)$. By taking exterior products we obtain the complex $k-$forms
$\Omega_{\mathbb{C}}^k(M)=\Upgamma(\Uplambda^k T_{\mathbb{C}}^*M)\simeq\Omega^k(M)\oplus i\Omega^k(M)$,
and tensoring by $T_{\mathbb{C}}M$ we get vector-valued complex $p-$forms,
$\Omega_{\mathbb{C}}^p(M;TM)=\Upgamma(\Uplambda^p T_{\mathbb{C}}^*M\otimes T_{\mathbb{C}}M)$.
We will denote $\Omega_{\mathbb{C}}(M)=\overset{\infty}{\underset{k=0}{\oplus}} \Omega^{k}_{\mathbb{C}}(M)$.

A derivation of degree $l\in\mathbb{Z}$ is a $\mathbb{C}-$endomorphism of $\Omega_{\mathbb{C}}(M)$
such that 
\begin{equation*}
D:\Omega^{k}_{\mathbb{C}}(M)\rightarrow\Omega^{k+l}_{\mathbb{C}}(M) ,
\end{equation*}
and, for any homogeneous form $\alpha$ and arbitrary $\beta$,
$$
D(\alpha\wedge\beta)=D(\alpha)\wedge\beta+(-1)^{l|\alpha|}\alpha\wedge\beta ,
$$
where $|\alpha|$ denotes the degree of $\alpha$. The vector space (and 
$\mathcal{C}^\infty (M)-$module) of such derivations is denoted 
$\mathrm{Der}^{l}\,\Omega_{\mathbb{C}}(M)$, and then
$\mathrm{Der}\,\Omega_{\mathbb{C}}(M)=
\oplus_{l\in\mathbb{Z}}\mathrm{Der}^{l}\,\Omega_{\mathbb{C}}(M)$ 
is a $\mathbb{Z}-$graded Lie algebra endowed
with the (graded) commutator of derivations
\begin{equation*}
[D_1,D_2]=D_1\circ D_2-(-1)^{|D_1||D_2|}D_2\circ D_1
\end{equation*} 
where $|D_1|$, $|D_2|$ denote the degree of $D_1$, $D_2$, respectively.

Any complex vector-valued $(k+1)-$form $K\in \Omega^{k+1}_{\mathbb{C}}(M;TM)$, defines
a derivation of degree $k$, called the insertion of $K$ and denoted $\imath_K$: if
$\omega\in \Omega^{p}_{\mathbb{C}}(M)$ and $Z_j\in\Upgamma(TM\oplus iTM)$, with $j=1,...,k+p$, then
\begin{align*}
& \imath_K\omega(Z_1,\cdots ,Z_{k+p}) \\
& =\sum_{\sigma\in S_{k+1,p-1}} \mathrm{sgn}(\sigma)\omega(K(Z_{\sigma(1)},\cdots,Z_{\sigma(k+1)}),Z_{\sigma(k+2)},\cdots,Z_{\sigma(k+p)}),
\end{align*}
where $S_{k+1,p-1}$ are shuffle permutations and $\mathrm{sgn}$ denotes the signature.

A derivation $D\in \mathrm{Der}\,\Omega_{\mathbb{C}}(M)$ such that it vanishes on functions, 
$D(f)=0$ for all $f\in \Omega^{0}_{\mathbb{C}}(M)$, is called a tensorial (or algebraic) 
derivation. They all are insertions.
\begin{pro}\label{insertions}
Let $D\in \mathrm{Der}^{k}\,\Omega_{\mathbb{C}}(M)$ be a tensorial derivation. 
Then, there exists a unique vector-valued complex form $K\in \Omega^{k+1}_{\mathbb{C}}(M;TM)$ 
such that
$$
D=\imath_K.
$$
\end{pro}
The $K$ in this proposition can be obtained simply by making $D$ act on complex valued
$1-$forms $\mathrm{d}f$.
If $K\in \Omega^{k+1}_{\mathbb{C}}(M;TM)$, $L\in \Omega^{l+1}_{\mathbb{C}}(M;TM)$, their 
corresponding insertions are $\imath _K\in \mathrm{Der}^{k}\,\Omega_{\mathbb{C}}(M)$, 
$\imath _L\in Der^{l}(\Omega_{\mathbb{C}}(M))$, so their graded commutator is again a derivation  
$[\imath _K,\imath _L]\in \mathrm{Der}^{k+l}\,\Omega_{\mathbb{C}}(M)$. Moreover, this new
derivations is obviously tensorial so, applying proposition \ref{insertions}, there exists a
unique vector-valued complex $(k+l+1)-$form determined by it.
\begin{defn}
Given $K\in \Omega^{k+1}_{\mathbb{C}}(M;TM)$, $L\in \Omega^{l+1}_{\mathbb{C}}(M;TM)$, their Richardson-Nijenhuis bracket is defined as the element $[K,L]_{RN}\in \Omega^{k+l+1}_{\mathbb{C}}(M;TM)$ such that
\begin{equation*}
[\imath_K,\imath _L]=\imath_{[K,L]_{RN}}.
\end{equation*}
\end{defn}

The bracket of vector fields on $M$ can be clearly extended by $\mathbb{C}-$linearity to
complex vector fields; the same occurs with the exterior differential $\mathrm{d}$, which can
be extended by $\mathbb{C}-$linearity to a derivation $\mathrm{d}\in Der^1\,\Omega_{\mathbb{C}}(M)$
(we will use the same notation for $\mathrm{d}$ and this extension). Then, given a 
$K\in  \Omega^{k+1}_{\mathbb{C}}(M;TM)$, as $\imath_K\in Der^{k}\,\Omega_{\mathbb{C}}(M)$ and
$\mathrm{d}\in Der^1\,\Omega_{\mathbb{C}}(M)$, their graded commutator will be again a derivation, called the Lie derivative along $K$,
\begin{equation*}
\mathcal{L}_K:=[\imath_K,\mathrm{d}]=\imath_K\circ\mathrm{d}-(-1)^{k}\mathrm{d}\circ\imath_K .
\end{equation*}
As a consequence of the Jacobi identity for the graded commutator of derivations, and the
nilpotency of $\mathrm{d}$, we get the following.
\begin{pro}
For any $K\in\Omega_{\mathbb{C}}(M;TM)$, the graded commutator of the derivations $\mathcal{L}_K$ 
and $\mathrm{d}$ vanishes.
\end{pro}

Thus, there exist tensorial derivations, of type $\imath_K$ for $K\in\Omega^{k}_{\mathbb{C}}(M;TM)$, and derivations commuting with the exterior differential $\mathrm{d}$, such as Lie
derivatives. The Fr\"olicher-Nijenhuis decomposition theorem states that any other derivation
is a sum of two of these.
\begin{thm}[Fr\"olicher-Nijenhuis]\label{teofn}
Let $D\in Der^{k}\,\Omega_{\mathbb{C}}(M)$. Then, there exists a unique couple $(K,L)$,
$K\in \Omega^{k}_{\mathbb{C}}(M;TM)$ and $L\in \Omega^{k+1}_{\mathbb{C}}(M;TM)$,
such that
\begin{equation*}
D=\mathcal{L}_K+\imath_L.
\end{equation*}
\end{thm}
A useful consequence is the following result.
\begin{cor}\label{corA}
Let $D\in \mathrm{Der}^{k}\,\Omega_{\mathbb{C}}(M)$. Then
\begin{enumerate}
\item $D$ is tensorial if and only if $K=0$.
\item $D$ commutes with $\mathrm{d}$ if and only if $L=0$.
\end{enumerate}
\end{cor}
\begin{exa}
The Fr\"olicher-Nijenhuis decomposition of the exterior differential is 
$$
\mathrm{d}=\mathcal{L}_{\mathrm{Id}},
$$
(where $\mathrm{Id}:TM\to TM$ is the identity morphism) as it should be for a derivation that 
commutes with $\mathrm{d}$.
\end{exa}
If $K\in \Omega^{k}_{\mathbb{C}}(M;TM)$ and $L\in \Omega^{l}_{\mathbb{C}}(M;TM)$, 
then $[\mathcal{L}_K,\mathcal{L}_L]$ is again a derivation and it commutes with $\mathrm{d}$, 
because of the Jacobi identity. Thus, according to Corollary \ref{corA}, it must be of the form 
$\mathcal{L}_{R} \in Der^{k+l}(\Omega_{\mathbb{C}}(M;TM))$ for a unique 
$R\in \Omega^{k+l}(\Omega_{\mathbb{C}}(M;TM))$.
\begin{defn}
Given $K\in \Omega^{k}_{\mathbb{C}}(M;TM)$ and $L\in \Omega^{l}_{\mathbb{C}}(M;TM)$, their Fr\"olicher-Nijenhuis bracket is the unique element 
$[K,L]_{FN}\in \Omega^{k+l}_{\mathbb{C}}(M;TM)$ such that
\begin{equation*}
[\mathcal{L}_K,\mathcal{L}_L]=\mathcal{L}_{[K,L]_{FN}}.
\end{equation*}
\end{defn}
Along with the Fr\"olicher-Nijenhuis bracket we have the notion of Nijenhuis torsion: if
$N\in\Omega^{1}_{\mathbb{C}}(M;TM)$ is a vector-valued $1-$form, it is defined as
the vector-valued $2-$form $T_N=\frac{1}{2}[N,N]_{FN}$. Explicitly, we have
\begin{equation}\label{torsion}
T_N(X,Y)=[NX,NY]-N[NX,Y]-N[X,NY]+N^2[X,Y],
\end{equation}
for any $X,Y\in\mathcal{X}(M)$.

\section{From Lie algebroids to cohomology operators}\label{sec2}
Suppose that $(E,q,[\![\, ,\, ]\!])$ is a Lie algebroid. Let us define an operator 
$D:\Upgamma(\Uplambda^p E^*)\to\Upgamma(\Uplambda^{p+1}E^*)$ by putting
\begin{equation}\label{defd}
\begin{cases}
Df(A)=qA(f) \\
D\alpha (A,B)=qA(\alpha (B))-qB(\alpha (A))-\alpha (\lie{A}{B}),
\end{cases}
\end{equation}
for any $f\in\mathcal{C}^\infty (M),\alpha\in\Upgamma(\Uplambda^1 E^*),A,B\in\Upgamma E$, and extending its
action to $\Upgamma(\Uplambda^\bullet E^*)$ as a derivation of $\mathbb{Z}-$degree $1$.
\begin{pro}\label{AC}
$D$ is a cohomology operator, that is,
$$
D^2=\frac{1}{2}[D,D]=0 ,
$$
where $[\, ,\,]$ denotes the graded commutator of derivations.
\end{pro}
%\begin{proof}
%First of all notice that $D^2$ is a derivation, thus a first-order differential operator and, as such, is
%determined by its action on forms of degree less than or equal to $1$. Developing the left-hand side of
%$(D^2f)(A,B)$ and $(D^2\alpha )(A,B,C)$, we get $0$. In the first case because $q$ is a morphism of Lie
%algebras (after a straightforward computation). In the second case, we have the cyclic sum
%$$
%(D^2\alpha )(A,B,C)= \underset{A,B,C}{\circlearrowright}(qA((D\alpha) (B,C))-(D\alpha) (\lie{A}{B},C)),
%$$
%and another long but straightforward computation shows that this reduces to the Jacobi identity for
%$[\![\, ,\, ]\!]$:
%$$
%(D^2\alpha )(A,B,C)=\alpha (\lie{\lie{A}{B}}{C}+\lie{\lie{B}{C}}{A}+\lie{\lie{C}{A}}{B}) .
%$$
%\end{proof}
This is a very well-known construction (see chapter 7 in \cite{mac05}, for instance),
sometimes called the \emph{De Rham differential of the Lie algebroid} (because one recovers the usual
exterior differential $\mathrm{d}$, when $E=TM$ endowed with the Lie bracket on vector fields and the
identity as anchor map).
The operator $D$ is a derivation of $\Upgamma(\Uplambda^\bullet E^*)$, so it can be decomposed \`a la
Fr\"olicher-Nijenhuis. Let us see how this can be done in a particular, but instructive, case.
\begin{exa}\label{tmcase}
When $E=TM$, we have a Lie algebroid $(TM,q,[\![\, ,\, ]\!])$ and a derivation $D\in\mathrm{Der}\Omega(M)$.
Recall (see \cite{Gra0601}) that given a bilinear operation $\circ$ on the sections of a vector bundle
$E$ (such as $[\![\, ,\, ]\!]$ on $\Upgamma TM$), and a bundle endomorphism $\varphi$ 
(such as $q:\Upgamma TM\to \Upgamma TM$), the 
\emph{contracted bracket}
of the operation is given by $A\circ_\varphi B=\varphi(X)\circ Y+X\circ \varphi(Y)-\varphi(X\circ Y)$,
for $A,B\in E$. In our case, the contracted bracket of the Lie
bracket of vector fields by the anchor map of the algebroid is
\begin{equation}\label{eq2}
[X,Y]_q=[qX,Y]+[X,qY]-q[X,Y].
\end{equation}
Now, a computation using the Fr\"olicher-Nijenhuis decomposition $D=\mathcal{L}_K +i_L$, where
$K\in\Omega^1(M;TM)$ and $L\in\Omega^2(M;TM)$, shows that
$$
K=q ,
$$
and
\begin{equation}\label{L}
L(X,Y)=[X,Y]_q-\lie{X}{Y}.
\end{equation}
Indeed, we have that $K$ is given by
$$
KX(f)=Df(X)=qX(f),
$$
for each $f\in\mathcal{C}^\infty(M)$, $X\in\mathcal{X}(M)$, and also
\begin{equation}\label{eq2b}
(D-\mathcal{L}_q)(\alpha )(X,Y)=(i_L\alpha )(X,Y)=\alpha (L(X,Y)).
\end{equation}
This is enough to characterize $i_L$, as it is determined by its action
on $1-$forms (on functions it vanishes by definition): Developing the left-hand side in \eqref{eq2b},
$$
(D\alpha)(X,Y)=qX(\alpha(Y))-qY(\alpha(X))-\alpha(\lie{X}{Y}),
$$
and, from the definition of $\mathcal{L}_q\alpha$,
\begin{align*}
&(\mathcal{L}_q\alpha)(X,Y) \\
&=qX(\alpha(Y))-\alpha([qX,Y])-qY(\alpha(X))+\alpha([qY,X])+\alpha(q[X,Y]) \\
&=qX(\alpha(Y))-qY(\alpha(X))-\alpha([qX,Y])-\alpha([Xq,Y])+\alpha(q[X,Y]),
\end{align*}
thus:
$$
\alpha(L(X,Y))=\alpha([qX,Y]+[X,qY]-q[X,Y]-\lie{X}{Y}).
$$
As this is valid for arbitrary $\alpha\in\Omega^1(M)$, $X,Y\in\mathcal{X}(M)$, we get \eqref{L}.
\end{exa}

The question now is how to obtain an analog result for an arbitrary Lie algebroid $(E,q,[\![\, ,\, ]\!])$,
where $D\in\mathrm{Der}\Upgamma(\Uplambda^\bullet E^*)$ and we do not have a Lie derivative 
operator at our disposal. We can proceed by choosing a connection $\nabla$ on $E$; then,
a proof analogous to that of the classical
Fr\"olicher-Nijenhuis theorem (see \cite{monmon88}), shows that $D$ must decompose itself as
\begin{equation}\label{eq1}
D=\nabla_K +\imath_L
\end{equation}
for certain tensor fields $K\in\Upgamma (E^*\otimes TM)$ and 
$L\in\Upgamma(\Uplambda^2 E^*\otimes E)$.
We briefly describe the proof of this result here, as we will make use of some of its intermediate steps below.
\begin{thm}
Let $\pi:F\to M$ be a smooth vector bundle over the differential manifold $M$.
Let $D\in\mathrm{Der}\se F)$, and let
$\nabla$ be a linear connection on $F^*$. Then, there exist unique tensor fields $K\in\Upgamma(F\otimes TM)$ and
$L\in\Upgamma(\Uplambda^2 F\otimes F^*)$, such that
\begin{equation}\label{decomp-delta}
D=\nabla_K +\imath_L.
\end{equation}
\end{thm}
\begin{proof}
Let $\alpha^1,\dots, \alpha^{k}\in\Gamma(F^*)$ be smooth sections. The map 
$f\rightarrow (Df)(\alpha^1,\dots, \alpha^{k})$, where $f\in \mathcal{C}^{\infty}(M)$, 
is a derivation so it defines a vector field on $M$, which we denote by 
$K(\alpha^1,\dots,\alpha^{k})$. The map from $\Gamma(F^*)\times\dots\times\Gamma(F^*)$ to 
$\mathfrak{X}(M)$ defined by 
$(\alpha^1,\dots,\alpha^{k})\rightarrow K(\alpha^1,\dots,\alpha^{k})$ is 
$\mathcal{C}^{\infty}(M)-$linear and skew-symmetric, therefore it defines a section 
$K\in\Gamma(\Lambda ^{k}F\otimes TM)$ that satisfies $\nabla_Kf=Df$ for every 
$f\in\mathcal{C}^{\infty}(M)$.\\
The operator $D-\nabla_K$ is a derivation of degree $k$ that acts trivially on 
$\mathcal{C}^{\infty}(M)$; therefore it is a $\mathcal{C}^{\infty}(M)-$linear endomorphism of 
$\Gamma(\Lambda F)$ which is determined by its action on the sections of degree $1$. Then, if 
$s\in \Gamma (F)$, the map $s\rightarrow (D-\nabla_K)s$ defines a morphism from 
$\Gamma(F)$ into $\Gamma(\Lambda^{k+1}F)$ so there is a section 
$L\in\Gamma(\Lambda^{k+1}F\otimes F^{*})$ such that $(D-\nabla_K)s=\imath_L s$. The operator 
$\imath_L$ is a derivation of degree $k$ that acts trivially on $\mathcal{C}^{\infty}(M)$
and as  $D-\nabla_K$ on the sections of $F$. Then, $D=\nabla_K+\imath_L$.
\end{proof}

The decomposition \eqref{eq1} above follows, of course, taking $F=E^*$. From the proof, we get that
$K:\Upgamma E\to \Upgamma TM$ is such that, for any $A\in \Upgamma E$ and $f\in\mathcal{C}^\infty(M)$,
\begin{equation}\label{exprK}
K(A)(f)=Df(A)=qA(f),
\end{equation}
so $K=q$, as before.

The following result will be useful when doing explicit computations.
\begin{lem}\label{lema1}
Let $K:E\to TM$ be a vector bundle morphism, and let $\nabla$ be a connection on $E$. If
$\alpha\in\Upgamma E^*$ and $A,B\in\Upgamma E$, then,
$$
(\nabla_K\alpha)(A,B)=KA(\alpha(B))-\alpha(\nabla_{KA}B)-KB(\alpha(A))+\alpha(\nabla_{KB}A).
$$
\end{lem}
\begin{proof}
Consider a basis of sections of $\Upgamma E$, $\{ s_a\}^r_{a=1}$ (with $r=\mathrm{rank}E$),
and let $\{\eta^b\}^r_{b=1}$ be the dual basis. Also, let $\{\partial_i\}^m_{i=1}$ be a local
basis of local vector fields on $M$ (with $\dim M=m$).
Then, we have the local expressions $A=A^as_a$, $B=B^bs_b$, and $K=K^j_c\eta^c\otimes\partial_j$. We can write
$$
\nabla_K =\nabla_{K^j_c\eta^c\otimes\partial_j}=K^j_c\eta^c\wedge\nabla_{\partial_j},
$$
and compute
\begin{align*}
&(\nabla_K \alpha)(A,B)=(K^j_c\eta^c\wedge\nabla_{\partial_j}\alpha)(A,B) \\
&=K^j_c (\eta^c(A)(\nabla_{\partial_j}\alpha)(B)-\eta^c(B)(\nabla_{\partial_j}\alpha)(A))\\
&=K^j_c (A^c(\nabla_{\partial_j}(\alpha (B))-\alpha(\nabla_{\partial_j}B))-B^c(\nabla_{\partial_j}(\alpha (A))-\alpha(\nabla_{\partial_j}A))) \\
&=\nabla_{K^j_cA^c\partial_j}(\alpha(B))-K^j_cA^c\alpha(\nabla_{\partial_j}B)-\nabla_{K^j_cB^c\partial_j}(\alpha(A))+K^j_cB^c\alpha(\nabla_{\partial_j}A).
\end{align*}
Notice that $KA=K^j_cA^c\partial_j$, so we can rearrange the preceding expression, using also the 
$\mathcal{C}^\infty(M)-$linearity of $\alpha$, as
\begin{align*}
(\nabla_K \alpha)(A,B)&=\nabla_{KA}(\alpha(B))-\alpha(\nabla_{KA}B)-\nabla_{KB}(\alpha(A))-\alpha(\nabla_{KB}A) \\
&=KA(\alpha(B))-KB(\alpha(A))+\alpha(\nabla_{KB}A-\nabla_{KA}B).
\end{align*}
\end{proof}
Now, we can find $L$ in the decomposition \eqref{eq1}. Let us start with:
$$
(D-\nabla_q)(\alpha )(A,B)=(\imath_L\alpha )(A,B)=\alpha (L(A,B)).
$$
Now, on the one hand,
$$
(D\alpha)(A,B)=qA(\alpha(B))-qB(\alpha(A))-\alpha(\lie{A}{B}),
$$
and on the other, from lemma \ref{lema1}
$$
\nabla_q\alpha(A,B)=qA(\alpha(B))-\alpha(\nabla_{qA}B)-qB(\alpha(A))+\alpha(\nabla_{qB}A).
$$
Thus,
$$
\alpha(L(A,B))=(D-\nabla_q)(\alpha)(A,B)=\alpha(\nabla_{qA}B-\nabla_{qB}A-\lie{A}{B}),
$$
hence:
\begin{equation}\label{tor}
L(A,B)=\nabla_{qA}B-\nabla_{qB}A-\lie{A}{B}.
\end{equation}
\begin{rem}
Let $(E,q,[\![\, ,\, ]\!])$ be a Lie algebroid over $M$. Let $V\to M$ another vector bundle over $M$.
An $E-$connection on $V$, $\delta$, is an $\mathbb{R}-$bilinear map $\delta:\Upgamma E\times\Upgamma V
\to \Upgamma V$, whose action is denoted $\delta(A,s)=\delta_A s$, such that, for any 
$f\in\mathcal{C}^\infty (M)$,
\begin{align*}
&\delta_{fA}s= f\delta_A s \\
&\delta_A (fs)= (qA)(f)s+f\delta_A s.
\end{align*}
An $E-$connection is simply an $E-$connection on $E$ itself (see \cite{fer02,mac05,blao06} for references 
on Lie algebroid connections and their applications). In this case, it is possible to define
the torsion of $\delta$ as the $E-$valued $2-$form $\mathrm{Tor}\delta\in\Omega^2(E;E)$ given by the
usual expression, but using the Lie algebroid bracket instead of the Lie one:
$$
\mathrm{Tor}\,\delta (A,B):=\delta_A B-\delta_B A-\lie{A}{B}.
$$
Notice that each linear connection $\nabla$ on $E$ determines an $E-$connection $\delta^\nabla$. Simply
define
\begin{equation}\label{deltanabla}
\delta^\nabla_A B :=\nabla_{qA}B.
\end{equation}
The torsion of this connection is then
$$
\mathrm{Tor}\,\delta^\nabla (A,B)=\nabla_{qA}B-\nabla_{qB} A-\lie{A}{B},
$$
which is precisely our expression \eqref{tor}.
\end{rem}

We summarize these calculations in the following result.
\begin{thm}
Let $(E,q,[\![\, ,\, ]\!])$ be a Lie algebroid and $D\in\mathrm{Der}\Upgamma(\Uplambda^\bullet E^*)$ be
the derivation defined in \eqref{defd}. Then, given any linear connection $\nabla$ on $E$, $D$ can be written
as
$$
D=\nabla_q +\imath_{L^\nabla} ,
$$
where $L^\nabla\in\Upgamma(\Uplambda^2 E^*\otimes E)$, the torsion of the $E-$connection $\delta^\nabla$,
is given by
$$
L^\nabla(A,B)=\nabla_{qA}B-\nabla_{qB}A-\lie{A}{B},
$$
for $A,B\in\Upgamma E$.
\end{thm}
An equivalent formulation of this result can be obtained in the particular case of a Lie algebroid
on the tangent bundle, which offers an alternative interpretation of the contracted bracket 
$[X,Y]_q$ in terms of the torsion of the $TM-$connection $\delta^\nabla$ in \eqref{deltanabla}.

Observe that if $E=TM$, then, given any symmetric linear connection $\nabla$ on $TM$,
the torsion of the corresponding $TM-$connection  $\delta^\nabla$
is given by
\begin{align*}
\mathrm{Tor}\, \delta^\nabla =& L^\nabla(X,Y) \\
=& \nabla_{qX}Y-\nabla_{qY}X-\lie{X}{Y} \\
=& [X,Y]_q+(\nabla_Y q)(X)-(\nabla_X q)(Y)-\lie{X}{Y},
\end{align*}
for any $X,Y\in\mathcal{X}(M)$.

Just apply that $\nabla$ is torsionless, and the properties of any covariant derivative.

\section{From cohomology operators to Lie algebroids}\label{sec3}
Suppose we have a smooth vector bundle $\pi:E\to M$, and a derivation 
$D\in\mathrm{Der}\Upgamma(\Uplambda E^*)$ such that
$D^2=0$. Then, we can define a mapping $q:\Upgamma E\to \Upgamma TM$ as follows,
$$
q(A)(f):=Df(A),
$$
for any $f\in\mathcal{C}^\infty (M)$, $A\in\Upgamma E$.\\
Let us define also a bracket on the sections of $E$ in the following way: if $A,B\in\Upgamma E$, then their bracket
$\lie{A}{B}$ is the section of $E$ characterized by
$$
\alpha (\lie{A}{B})=D(\alpha(B))(A)-D(\alpha(A))(B)-D\alpha(A,B),
$$
for any $\alpha\in\Upgamma E^*$. The following result is well-known.
\begin{pro}\label{pro2}
The triple $(E,q,[\![\, ,\, ]\!])$ is a Lie algebroid.
\end{pro}

\begin{rem} It follows from
proposition \ref{AC} and proposition \ref{pro2}, that there exists a one to one
correspondence between Lie algebroids on E and cohomology  operators in
$\Upgamma(\Uplambda E^*)$, that is, derivation $D$ of degree one on
$\Upgamma(\Uplambda E^*)$ such that $D^2=0$. For various formulations of this
results see, for example, \cite{kosmag90} and \cite{cramoe04}.
\end{rem}

Let us delve into the structure of this Lie algebroid when $E=TM$. In this case, last proposition can be reformulated: Let $D$ be a derivation, $D\in\mathrm{Der}\Upgamma(\Uplambda^\bullet T^*M)$, such that $D^2=0$, then it has a Fr\"olicher-Nijenhuis decomposition $D=\mathcal{L}_K+\imath_L$, where $K\in\Omega^1(M;TM)$ and $L\in\Omega^2(M;TM)$. Then, the triple $(E,q,[\![\, ,\, ]\!])$ is a Lie algebroid with 
$$q=K$$
and
\begin{equation}\label{eq3}
\lie{X}{Y}=[X,Y]_K-L(X,Y),
\end{equation}
where the contracted bracket is given in \eqref{eq2}.\\
Formula \eqref{eq3} can be deduced by observing that $D$ is precisely the operator associated to the Lie algebroid by the 
construction in the previous section. In particular, that tells us that the distribution $\mathrm{Im}K$ is 
involutive, because (as the anchor map of a Lie algebroid is a Lie algebra morphism) 
$[KX,KY]=K\lie{X}{Y}\in\mathrm{Im}K$. Notice (see \cite{Gra0601}) that when the Nijenhuis
torsion of $K$ vanishes, then $[\, ,\, ]_K$ is a Lie bracket.

Taking the action of $K$ on both sides of \eqref{eq3}, we get
$$
KL(X,Y)=K[KX,Y]+K[X,KY]-K^2[X,Y]-K\lie{X}{Y}.
$$
But, as $K=q$ is a Lie algebroid anchor, it is a Lie algebra morphism, so
$$
[KX,KY]-K[KX,Y]-K[X,KY]+K^2[X,Y]=-KL(X,Y),
$$
or, more succintly, in terms of the Nijenhuis torsion of $K$,
\begin{equation}\label{eqKL}
T_K(X,Y)=-KL(X,Y).
\end{equation}
This expression allows us to prove the following result:
Let $D\in\mathrm{Der}^1\Omega(M)$ be such that $D^2=0$ and
its Fr\"olicher-Nijenhuis decomposition has the form $D=\mathcal{L}_K$ (i.e, $L=0$).
\begin{pro}\label{TN}
The derivation $D=\mathcal{L}_K$ induces a Lie algebroid structure on (sections of) $TM$ with
anchor map $q=K$ and bracket
$$
[X,Y]_K=[KX,Y]+[X,KY]-K[X,Y],
$$
\emph{if and only if} $K$ is integrable (i.e, $T_K=0$).
\end{pro}
\begin{proof}
The condition $T_K=0$ follows from \eqref{eqKL} above. On the other hand, the
sufficiency is guaranteed by the identities $D^2=\frac{1}{2}[D,D]$ and
$[\mathcal{L}_K,\mathcal{L}_L]=\mathcal{L}_{[K,L]_{FN}}$ and by proposition \ref{pro2}.
\end{proof}
\begin{rem}
The `if' part of this proposition was given as theorem 3.7 in \cite{gra1301} and as exercise 40 in \cite{crafer11}.
\end{rem}
%\begin{exa}\label{multiexa}
Recall that, given a bundle map
$N:TM\rightarrow TM$, when $T_N=0$, $N$ is called a Nijenhuis tensor. As
every Nijenhuis tensor induce a Lie algebroid, then, any complex structure ($K^2=-I$), tangent structure ($K^2=0$) or product structure ($K^2=I$), defines a Lie 
algebroid on $TM$.
%\end{exa}
Let us see what happens when a term $L\neq 0$ is non trivial. Still demanding that
$D=\mathcal{L}_K+\imath_L$ be of square zero; we will do this in two steps, 
assuming in the first one that $K$ is invertible.
If $K:\Upgamma TM\rightarrow \Upgamma TM$ is an invertible endomorphism, the vector-valued $2-$form
$L\in\Omega^2 (M;TM)$ is determined by 
\begin{equation}
L(X,Y):=-K^{-1} T_K(X,Y),
\end{equation}
for any $X, Y\in\Upgamma TM$, and then, 
the bracket $[\![\;,\;]\!]:\Upgamma TM\rightarrow \Upgamma TM$ is 
\begin{equation}\label{corchete}
[\![X,Y]\!]:=[X,Y]_K+K^{-1} T_K(X,Y),
\end{equation}
for all $X, Y\in\Upgamma TM$.
\begin{lem}\label{B}
The bracket \eqref{corchete} satisfies
$$
[\![X,Y]\!]=K^{-1}[KX,KY].
$$
\end{lem}
\begin{proof}
It is a straightforward computation:
\begin{align*}
[\![X,Y]\!]=&[X,Y]_K+K^{-1} T_K(X,Y)\\
=& [KX,Y]+[X,KY]-K[X,Y]\\
&+K^{-1}([KX,KY]-K[KX,Y]-K[X,KY]+K^2[X,Y])\\
=&  K^{-1}[KX,KY]
\end{align*}
\end{proof}
\begin{pro}\label{C}
The triple $(TM,K,[\![\, ,\,]\!])$, with the bracket defined in \eqref{corchete}, 
is a Lie algebroid.
\end{pro}
\begin{proof}
That $(\Upgamma TM,[\![\;,\;]\!])$ is a Lie algebra is immediate (Jacobi's identity
results by applying Lemma \ref{B}), and $K$ is a morphism of 
$\mathcal{C}^{\infty}(M)-$modules, so we only need to check the Leibniz rule:
$$
\begin{array}{lll}
[\![X, fY]\!]&=&[X,fY]_K+K^{-1}T_K(fX,Y)\\
&=&[KX,fY]+[X,fKY]-K[X,fY]+fK^{-1}T_K(X,Y)\\
&=& f[X,Y]_K+KX(f)Y+X(f)KY-X(f)KY+fK^{-1} T_K(X,Y)\\
&=& f[\![X, Y]\!]+KX(f)Y.
\end{array}
$$
\end{proof}
Every Lie algebroid induced by an invertible endomorphism is isomorphic to the trivial Lie algebroid:
\begin{pro}\label{trivial}
The Lie algebroids $(TM,K,[\![\, ,\,]\!])$ (given by proposition \ref{C}) and the trivial one
$(TM,Id,[\, ,\,])$ are isomorphic.
\end{pro}
\begin{proof}
The desired isomorphism is $\phi=K^{-1}:\Upgamma TM\rightarrow \Upgamma TM$, since
$$
K\circ \phi=K\circ K^{-1}=Id\,,
$$
and
$$
\phi([X,Y])=K^{-1}([X,Y])=K^{-1}([K\circ K^{-1} (X),K\circ K^{-1}(Y)])=[\![\phi(X) ,\phi(Y)]\!]
$$
by lemma \ref{B}.
\end{proof}

To summarize, we have the following result which, in particular, applies to almost-complex or almost-product 
structures.
\begin{thm}\label{thm2}
Let $K:TM\to TM$ be an invertible bundle endomorphism. Then, a derivation of the form
$$
D=\mathcal{L}_K+\imath_L\in\mathrm{Der}^1\Omega(M),
$$ 
(with $L\in\Omega^2(M;TM)$) has vanishing square if and only if
$$
L=-K^{-1}T_K .
$$
In this case, the Lie algebroid structure on $TM$ determined by $D$ is isomorphic to the
trivial one.
\end{thm}
In the next step, we consider the general situation of a derivation of $\mathbb{Z}-$degree $1$, 
$D=\mathcal{L}_K+i_L$, where $K$ is not necessarily invertible.
We ask ourselves under which conditions on $K$ and $L$, it defines a Lie algebroid structure.
From the Fr\"olicher-Nijenhuis decomposition $D=\mathcal{L}_K+i_L$ and the general properties of the
Fr\"olicher-Nijenhuis and Richardson-Nijenhuis brackets (see \cite{ksm94,mic87,niri67}), we get:
\begin{align*}
D^2&=\frac{1}{2}[D,D]=\frac{1}{2}[\mathcal{L}_K+\imath_L,\mathcal{L}_K+\imath_L]\\
&=\frac{1}{2}\mathcal{L}_{[K,K]_{FN}}+\imath_{[K,L]_{FN}}+\mathcal{L}_{\imath_LK}+
\frac{1}{2}\imath_{[L,L]_{RN}},
\end{align*}
so $D^2=0$ is equivalent to the conditions
\begin{equation}\label{eqBrackets}
\begin{dcases}
\frac{1}{2}[K,K]_{FN}+\imath_LK=0\\
[K,L]_{FN}+\frac{1}{2}[L,L]_{RN}=0.
\end{dcases}
\end{equation}
The first of these is already known, it comes from the imposition that the Lie algebroid bracket be
of the form $\lie{X}{Y}=[X,Y]_K-L(X,Y)$ and is nothing more than equation \eqref{eqKL}.
A straightforward computation shows that, in fact, it is also equivalent
to $D^2f=0$, for any $f\in\mathcal{C}^\infty(M)$. The second condition is much more difficult to satisfy.
In the following sections we will construct some solutions out of structures on the tangent bundle
of a manifold $M$ with geometric relevance.

\section{Lie algebroids associated to idempotent endomorphisms}\label{sec4}
Let $TM\to M$ be the tangent bundle over the connected manifold $M$, 
and let $N:TM\to TM$ be an idempotent 
endomorphism (that is, such that $N^2=N$). Then, $N$ has locally constant rank
and $\ker N$, $\mathrm{Im}\,N$ are
vector (sub)bundles such that $TM=\ker N\oplus\mathrm{Im}\,N$ (see \cite{hus08}). We will assume that $\mathrm{Im}\,N$ is an involutive distribution and write down proofs in the
real case for simplicity, but the results are valid in the complex case as well.
\begin{lem}\label{ct}
Under the above assumptions, the following holds: 
\begin{equation}
N[NX,NY]=[NX,NY].
\end{equation}
\end{lem}
\begin{proof}
As $\mathrm{Im}N$ is involutive, there exists a $Z\in \Upgamma TM$ such that $[NX,NY]=NZ$, and by applying $N$
to both sides of this equation,
\begin{equation*}
N[NX,NY]=N^2Z=NZ=[NX,NY]
\end{equation*}
\end{proof}
Consider now the Nijenhuis torsion of $N$ \eqref{torsion}.
In our case, it can be rewritten as
\begin{equation*}
T_N(X,Y)=[NX,NY]-N[NX,Y]-N[X,NY]+N[X,Y].
\end{equation*}
By applying $N$ to both sides, we get
\begin{equation*}
N T_N(X,Y)=N[NX,NY]-N^2[NX,Y]-N^2[X,NY]+N^2[X,Y],
\end{equation*}
and, using lemma \ref{ct} along with the property $N^2=N$,
\begin{equation}\label{NTn}
N T_N(X,Y)=[NX,NY]-N[NX,Y]-N[X,NY]+N[X,Y]=T_N(X,Y).
\end{equation}
This suggest to define the vector-valued $2-$form 
\begin{equation}\label{LTN}
L:=-T_N.
\end{equation}
The pair $(N,L)$ is a solution to equations \eqref{eqBrackets}.
\begin{pro}\label{pro5.3}
Let $N: TM\rightarrow TM$ be an idempotent endomorphism with $\mathrm{Im}\,N$ involutive, and $L\in\Omega^2(M;TM)$
defined by \eqref{LTN}. Then $NL=-T_N$, also
\begin{equation}
\begin{dcases}
[N,L]_{FN}=0 \\
[L,L]_{RN}=0
\end{dcases}
\end{equation}
\end{pro}
\begin{proof}
The first affirmation is a direct consequence of definition \eqref{LTN}.\\
As $L=-\frac{1}{2}[N,N]_{FN}$, and $(\Omega^{*}(M;TM),[\;,\;]_{FN})$ is a graded Lie algebra (see \cite{mic87}),
we have
\begin{equation*}
[N,L]_{FN}=-\frac{1}{2}[N,[N,N]_{FN}]_{FN}=0.
\end{equation*}
Moreover, as $L\in\Omega^2(M;TM)$, it follows that $[L,L]_{RN}\in\Omega^3(M;TM)$, $i_L\in Der^1\Omega(M)$,
and $i_{[L,L]_{RN}}\in Der^2\Omega(M)$; thus, $i_{[L,L]_{RN}}$ is completely characterized by its action on 
$0-$forms (smooth functions) and $1-$forms. But $i_{[L,L]_{RN}}$ is a tensorial derivation, so it vanishes on 
$\mathcal{C}^\infty(M)$. Now, given an $\alpha\in\Omega^1(M)$ it is $i_{[L,L]_{RN}}\in \Omega^3(M)$ so,
whenever $X,Y,Z\in\Upgamma TM$, 
\begin{align*}
(i_{[L,L]_{RN}}\alpha)(X,Y,Z)&=\alpha([L,L]_{RN}(X,Y,Z))\\
&=2i_L^2\alpha(X,Y,Z)\\
&=2\circlearrowright i_L\alpha(L(X,Y),Z)\\
&=2\circlearrowright\alpha(L(L(X,Y),Z))\\
&=2\circlearrowright\alpha(T_N(T_N(X,Y),Z))=0,
\end{align*}
(here $\circlearrowright$ denotes cyclic sum in $(X,Y,Z)$) because 
\begin{align*}
&T_N(T_N(X,Y),Z)\\
=&[N T_N(X,Y),NZ]-N[N T_N(X,Y),Z]-N[T_N(X,Y),NZ]+N^2[T_N(X,Y),Z]\\
=&N[N T_N(X,Y),NZ]-N[T_N(X,Y),Z]-N[T_N(X,Y),NZ]+N[T_N(X,Y),Z]\\
=&N[T_N(X,Y),NZ]-N[T_N(X,Y),Z]-N[T_N(X,Y),NZ]+N[T_N(X,Y),Z]=0,
\end{align*}
where we have used lemma \ref{ct}, equation \eqref{NTn}, and the property $N^2=N$.
\end{proof}
Combining \eqref{eqBrackets} with proposition \ref{pro5.3}, we arrive at the following result
\begin{thm}\label{Lct}
Let $N\in\Omega^1(M;TM)$ be such that $N^2=N$ and $\mathrm{Im}\,N$ is involutive. Then,
defining $L$ as in \eqref{LTN}, the derivations given by $D_1=\mathcal{L}_N+\imath_L$,
and $D_2=\imath_L \in\mathrm{Der}^1\Omega (M)$ are cohomology operators, that is, 
$D_1^2=0=D_2^2$.
\end{thm}
Note that when $N$ is
a Nijenhuis tensor, automatically  $\mathrm{Im}\,N$ is involutive, since in this
case $[N(X),N(Y)]=N([N(X),Y]+[X,N(Y)]-N([X,Y]))$ for every $X,Y\in
\Gamma TM$. And if it holds that $\mathrm{Ker}\,N$ and $\mathrm{Im}\,N$ are involutive, it follows
that $N$ is Nijenhuis tensor.
\begin{cor}\label{Lctc}
Let $N\in\Omega^1(M;TM)$ be such that $N^2=N$ and $T_N=0$. Then, the derivation given by $D=\mathcal{L}_{Id-N} \in\mathrm{Der}^1\Omega (M)$ is a
cohomology operator, that is, $D^2=0$.
\end{cor}
\begin{proof}
First we observe that $(Id-N)^2=Id-N$, moreover using the identity $[Id,N]_{FN}=-[N,Id]_{FN}$, we get  
\begin{align*}
T_{\mathrm{Id}-N}=&[Id-N,Id-N]_{FN}\\
=&[Id,Id]_{FN}-[N,Id]_{FN}-[Id,N]_{FN}+[N,N]_{FN}\\
=&0
\end{align*}
Then simply apply theorem \ref{Lct}.
\end{proof}
Applying formula \eqref{eq3}, the Lie algebroid induced by the operator $D_1=\mathcal{L}_N+\imath_L$ of Theorem \ref{Lct}
can be explicitly described as follows.
\begin{thm}\label{exampletm}
Let $N\in\Omega^1(M;TM)$ be such that $N^2=N$ and $\mathrm{Im}\,N$ is involutive. Then, there exists a Lie algebroid structure $(TM,q,[\!|\, ,\,|\!])$ with anchor map $q=N$ and bracket
\begin{equation*}
[\!|X,Y|\!]=[X,Y]_N+T_N(X,Y).
\end{equation*}
\end{thm}
We remark that the bracket on $TM$ constructed this way, is given by the sum of the
deformed bracket $[\cdot ,\cdot ]_N$ \emph{and} the torsion $T_N$. Another useful
expresion for this bracket is
\begin{equation}\label{bracketfull}
[\!|X,Y|\!]=[NX,NY]+(\mathrm{Id}-N)([NX,Y]+[X,NY])\,.
\end{equation} 

\section{Relation with the foliated exterior differential}\label{sec5}
Let $\mathcal{F}$ be a regular foliation on a manifold $M$ and $T\mathcal{F}\subset TM$ its tangent bundle. Denote by $j:T\mathcal{F}\hookrightarrow TM$ the  canonical injection. To any such a
foliation we can associate a Lie algebroid in two equivalent ways.
\begin{enumerate}[(a)]
\item The space of tangent sections $\Upgamma (T\mathcal{F})$ is closed with respect to the
Lie bracket of vector fields on $M$, by the integrability of $\mathcal{F}$, so it carries a
natural (foliated) bracket $[\cdot ,\cdot ]_{\mathcal{F}}$ which is just the restriction of the
Lie bracket to $\mathcal{F}$. Then,
$$
(E=T\mathcal{F},[\cdot ,\cdot ]_{\mathcal{F}},q=j:T\mathcal{F}\hookrightarrow TM)\,,
$$
is a Lie algebroid, called the Lie algebroid of the
foliation $\mathcal{F}$ (see, for example, \cite{fer02}). In this case the corresponding cohomology operator
is the foliated exterior differential 
$\mathrm{d}_{\mathcal{F}}:
\Upgamma(\Uplambda^p T^*\mathcal{F})\to\Upgamma(\Uplambda^{p+1} T^*\mathcal{F})$, which
gives rise to the foliated DeRham cohomology.

\item\label{itemb} Let $\mathbb{H}$ be an Ehresmann connection on $(M,\mathcal{F})$ complementary to $T\mathcal{F}$,
that is, $TM=\mathbb{H}\oplus\mathcal{F}$. Denote by $\gamma$ the associated projection on
$\mathcal{F}$, $\gamma :\mathbb{H}\oplus\mathcal{F}\to\mathcal{F}$, so 
$\mathbb{H}=\ker\gamma$ and $\mathcal{F}=\mathrm{Im}\gamma$. It follows that
$\gamma$ is a vector-valued $1-$form, and then the curvature of the connection is defined as
the Nijenhuis torsion
$$
R=\frac{1}{2}[\gamma ,\gamma]_{FN}=T_\gamma\,.
$$
On the whole $TM$ we can construct the Lie algebroid structure given by Theorem \ref{exampletm}.
In this case, the explicit expression of the bracket can be easily computed from 
\eqref{bracketfull}:
\begin{align*}
&[X,X']_\gamma = 0\,,\mbox{ for all }X,X'\in\mathbb{H}\\
&[X,Y]_\gamma = [X,Y]-\gamma [X,Y]\,,\mbox{ for all }X\in\mathbb{H},Y\in\Upgamma (T\mathcal{F})\\
&[Y,X]_\gamma = [Y,X]-\gamma [Y,X]\,,\mbox{ for all }X\in\mathbb{H},Y\in\Upgamma (T\mathcal{F})\\
&[Y,Y']_\gamma = [Y,Y']\,,\mbox{ for all }Y,Y'\in\Upgamma (T\mathcal{F})\,.
\end{align*}
The anchor map, as we know, is given by $q=\gamma$. Obviously, this algebroid structure coincides
with the previous one when restricted to $\Upgamma (T\mathcal{F})$, independently of the
chosen connection $\gamma$.
\end{enumerate}

In what follows, we offer an interpretation of the foliated cohomology within the
framework described in item \eqref{itemb}. 
The splitting induced by $\gamma$, $TM=\mathbb{H}\oplus T\mathcal{F}$, define a bigrading
on $\Omega (M)$. A differential form $\omega$ on a foliated manifold is said to be of 
type $(p,q)$ if it has degree $p+q$ and $\omega(X_1,\dots,X_{p+q})=0$ 
whenever the arguments contain more than $q$ vector fields in $\Upgamma (T\mathcal{F})$, or
more than $p$ in $\mathbb{H}$. According to this bigrading, we have a decomposition
of the exterior differential on $M$,
\begin{equation}\label{ddesc}
\dd = \dd_{1,0}+\dd_{2,-1}+\dd_{0,1}\,,
\end{equation}
(see \cite{va73}), where $\dd_{i,j}:\Omega^{p,q}(M)\to\Omega^{p+i,q+j}$. 
\begin{rem}\label{equivalence}
Denoting by $h:\se T^*\mathcal{F})\to\se\mathbb{H}^0)$ the natural
identification, we notice that $\dd_{0,1}$ is a $\gamma-$dependent extension of the
foliated exterior differential, that is, 
$$
(\dd_{0,1}\circ h)\,\omega = (h\circ \dd_{\mathcal{F}})\,\omega\,,
$$
for all $\omega\in\se T^*\mathcal{F})$.
\end{rem}

The Fr\"olicher-Nijenhuis decomposition of the operators appearing in the decomposition
\eqref{ddesc} can be readily computed from the results in section \ref{sec1}:
\begin{align*}
&\mathrm{d}_{1,0} =\, \mathcal{L}_{\mathrm{Id}-\gamma} +\imath_{2R} \\
&\mathrm{d}_{0,1} =\, \mathcal{L}_{\gamma} - \imath_{R} \\
&\mathrm{d}_{2, -1} =\, -\imath_{R}\, .
\end{align*}

From \eqref{ddesc} and $d^2=0$, it follows that

$$
\mathrm{d}^2_{0,1}=0=\mathrm{d}^2_{2, -1}\,.
$$

Moreover, the derivation $\mathrm{d}_{1,0}$ is a cohomology operator if and only if
the curvature of the  connection $\gamma$ vanishes.

\begin{thm}
The cohomology operator associated to the Lie algebroid on $TM$ described in item \eqref{itemb} above, is given by
$$
D=\dd_{0,1}\,.
$$
Moreover, the Lie algebroid of the foliation $\mathcal{F}$ and the restriction to $T\mathcal{F}$ of
the Lie algebroid associated to a connection $\gamma$ on $TM$, coincide. Thus,
the complexes associated to the cohomology operators $\dd_{\mathcal{F}}$ and $\dd_{0,1}|_{\se\mathbb{H}^0)}$
are isomorphic. 
\end{thm}
\begin{proof}
The first statement is just a consequence of Theorem \ref{AC} and example \ref{tmcase}. The second one follows from preceding results, along with remark \ref{equivalence}.
\end{proof}
\begin{rem}
Notice that, once this equivalence has been established, given the Lie algebroid structure
$N:TM\to TM$, with $N$ idempotent, we could define the associated foliated cohomology independently of any regularity condition on $\mathcal{F}$, just using the decomposition
$$
\mathrm{d}=\mathcal{L}_{Id}= \left(\mathcal{L}_{\mathrm{Id}-\gamma} +2\imath_{R}\right)
+\left( \mathcal{L}_{\gamma} - \imath_{R}\right) -\imath_{R}\,.
$$ 
\end{rem}

\section{Complex Lie algebroids associated to complex structures}\label{sec6}
Recall that an almost-complex structure on a manifold $M$ is a vector-valued $1-$form $J\in\Omega^1(M;TM)$
such that
$J^2=-\mathrm{Id}_{TM}$. In order to diagonalize the endomorphism $J$, 
we work in the complexified tangent bundle
$T^\mathbb{C}M$, and extend by
$\mathbb{C}-$linearity all real endomorphisms and differential operators on $TM$ 
(with a little abuse of notation, 
we will denote these extensions by the same symbols as their real counterparts). 
The Lie bracket of vector fields can also be extended by $\mathbb{C}-$linearity. 

The almost-complex structure $J$ is said to be integrable if its Nijenhuis torsion vanishes,
that is, if for every $X,Y\in \mathcal{X}(M)$,
$$
T_J(X,Y)=\frac{1}{2}[J,J]_{FN}(X,Y)=[JX,JY]-J[JX,Y]-J[X,JY]-[X,Y]=0.
$$
The Newlander-Nirenberg theorem states that this happens if and only if $M$ has he structure
of a complex manifold.

From the condition $J^2= -\mathrm{Id}_{TM}$, we know that $J$ has the eigenvalues $\pm i$. 
If we define the  projection operators
\begin{equation*}
p^{\pm}:=\frac{1}{2}(\mathrm{Id}\mp iJ):T^{\mathbb{C}}M\rightarrow T^{\mathbb{C}}M,
\end{equation*}
we get the usual properties 
\begin{align*}
(p^{\pm})^2&=p^{\pm} \\
p^++p^-&=\mathrm{Id}_{T^{\mathbb{C}}M} \\
p^+\circ p^-&=0=p^-\circ p^+ \, ,
\end{align*}
which determine the subbundle decomposition $T^{\mathbb{C}}M=T^+M\oplus T^-M$, where 
\begin{equation*}
\Upgamma T^{\pm}M=\{Z\in \Upgamma T^{\mathbb{C}}M:JZ=\pm iZ\}.
\end{equation*}
The elements of $\Upgamma T^+M$ are called holomorphic vector fields, and those of $\Upgamma T^-M$
anti-holomorphic.
Notice that $\mathrm{Im}p^\pm =T^\pm M$, and $\ker p^\pm =T^\mp M$, 
and that if $Z\in \Upgamma T^+M$,
then its complex
conjugate $\bar{Z}\in \Upgamma T^-M$ (and vice versa).

We will need the following technical results.
\begin{lem}
Let $J\in\Omega^1(M;TM)$ be an almost-complex structure on $M$, 
and let $T_J$ be its Nijenhuis torsion.
If $T_J=0$, then its complex extension also satisfies $\mathcal{T}_J=0$.
\end{lem}
\begin{proof}
Let $Z=X+iY, W=X'+iY'$. A straightforward computation shows that
\begin{align*}
\mathcal{T}_J(Z,W)&=[JZ,JW]-J[JZ,W]-J[Z,JW]-[Z,W]\\
&=T_J(X,X')-T_J(Y,Y')+i(T_J(Y,X')+T_J(X,Y'))
\end{align*}
\end{proof}
\begin{pro}\label{T}
An almost-complex structure $J\in\Omega^1(M;TM)$ is integrable if and only if 
$[T^+M,T^+M]\subset T^+M$,
that is, the distribution defined by the holomorphic vector fields is involutive.
\end{pro}
\begin{proof}
Consider first the case of $J\in\Omega^1(M;TM)$ integrable. 
By the preceding lemma, if $T_J= 0$, then also $\mathcal{T}_J=0$. 
If $Z,W$ are holomorphic vector fields, we have
$JZ=iZ$ and $JW=iW$, so
\begin{align*}
0=\mathcal{T}_J(Z,W)&=[iZ,iW]-J[iZ,W]-J[Z,iW]-[Z,W]\\
&=-[Z,W]-iJ[Z,W]-iJ[Z,W]-[Z,W]\\
&=-2([Z,W]+iJ[Z,W]),
\end{align*}
that is $J[Z,W]=-\frac{1}{i}[Z,W]=i[Z,W]$, so $[Z,W]$ is holomorphic.\\
Reciprocally, assume that $[Z,W]$ is a holomorphic vector field whenever $Z,W$ are. Given arbitrary
vector fields $X,Y\in\Upgamma TM$, we can rewrite them in the form $X=Z+\bar{Z}$, $Y=W+\bar{W}$ for
some $Z,W\in \Upgamma T^+M$. Then, it is readily proved that 
$\mathcal{T}_J(Z,\bar{W})=\mathcal{T}_J(\bar{Z},W)=0$,
and
\begin{equation*}
T_J(X,Y)=\mathcal{T}_J(Z,W)+\mathcal{T}_J(\bar{Z},\bar{W}).
\end{equation*}
But for holomorphic vector fields we already know that $\mathcal{T}_J(Z,W)=-2([Z,W]+iJ[Z,W])$, and,
as $[Z,W]$ is holomorphic by assumption, $J[Z,W]=i[Z,W]$. 
Therefore $\mathcal{T}_J(Z,W)=-2([Z,W]+i^2[Z,W])=0$.
A similar calculation shows that $\mathcal{T}_J(\bar{Z},\bar{W})=0$, 
and we conclude that $J$ is integrable.
\end{proof}

We are now ready to construct a complex Lie algebroid canonically
associated to a complex manifold.
Recall (see \cite{wei07}) that a complex Lie algebroid over the manifold $M$
is given by a complex vector bundle $E$ over $M$ endowed with a complex Lie
algebra structure on its space of sections $\Upgamma E$,
together with a bundle map $q:E\to T^\mathbb{C}M$ (the anchor map) satisfying
the Leibniz rule
$$
[A ,fB ]=f[A ,B ]+(qA)(f)B ,
$$
for any $A ,B \in\Upgamma E$, and $f:M\to\mathbb{C}$ smooth. 
The main example of a complex Lie algebroid
is given by the inclusion $q:T^+M\hookrightarrow T^\mathbb{C}M$
of the subbundle of holomorphic vector fields
on a complex manifold. Notice that the vector bundle in this case,
$E=T^+M$, is not $T^\mathbb{C}M$.
\begin{rem}
We want to make use of the property $(p^+)^2=p^+$ of the projection operator 
$p^+=\frac{1}{2}(\mathrm{Id}-iJ)$,
and the results in section \ref{sec4} to construct an algebroid where $p^+$ 
will be the anchor map, and the whole $T^\mathbb{C}M$ the corresponding bundle.
However, as $\mathrm{Im}p^+$ must be involutive, proposition \ref{T} forces us
to restrict ourselves to \emph{complex} manifolds.
\end{rem}
Our main result is, then, the following.
\begin{thm}\label{CA}
Let $J\in\Omega^1(M;TM)$ be a complex structure on the manifold $M$.
There exists a complex Lie algebroid
structure on $M$, $(T^{\mathbb{C}}M, q, [\![\, ,\,]\!])$,
where the anchor map is $q=p^+=\frac{1}{2}(Id-iJ)$,
and the bracket
$$
[\![Z,W]\!]=[Z,W]_{p^+}=\frac{1}{2}([Z,W]-i[Z,W]_J).
$$
\end{thm}
\begin{proof}
First of all, notice that given a complex structure $J$ on $M$, there is a relation
$$
\mathcal{T}_{p^+}=-\frac{1}{4}\mathcal{T}_J,
$$
which can be proved by a direct computation. According to the results on section \ref{sec4},
the idempotent endomorphism $p^+$ will induce a Lie algebroid $T^\mathbb{C}M\to T^\mathbb{C}M$
where the anchor map is precisely $q=p^+$, and the bracket
$
[\![Z,W]\!]=[Z,W]_{p^+}-\mathcal{T}_{p^+}.
$
However, the above comment implies that $\mathcal{T}_{p^+}=0$, so the bracket is simply
$$
[\![Z,W]\!]=[Z,W]_{p^+}. 
$$
A further straightforward computation shows that
$$
[Z,W]_{p^+}=\frac{1}{2}([Z,W]-i[Z,W]_J).
$$
\end{proof}
Notice that this Lie algebroid is different to that obtained by applying proposition \ref{TN},
where the resulting bracket would be $[\!|Z,W|\!]=[Z,W]_J$ (that is, the $\mathbb{C}-$linear
extension of $[\, ,\, ]_J$ to $T^\mathbb{C}M$).
\begin{rem}
The proof extends trivially to the case of an endomorphism $J\in\Omega^1(M;TM)$ 
such that $J^2=-\epsilon^2Id_{TM}$ with $\epsilon \in\mathbb{R}-\{0\}$.
\end{rem}

\section{Lie algebroids associated to product structures}\label{sec7}
Consider now an almost-product structure on the manifold $M$, that is, a vector bundle endomorphism
$P:TM\rightarrow TM$ such that $P^2=\mathrm{Id}_{TM}$. It has locally constant rank and it defines
two projection operators associated to its eigenvalues $\lambda=\pm 1$, 
$p^{\pm}:=\frac{1}{2}(\mathrm{Id}\pm P)$.
We have a decomposition of the tangent bundle analogous to that of the complex case,
$TM=T^+M\oplus T^-M$, where, this time, $T^\pm M=\mathrm{Im}p^\pm$.

A basic result is that the complementary distributions $T^\pm M$ are integrable if and only if $P$
is integrable, in the sense of having vanishing Nijenhuis torsion (see \cite{leorod89}), that is, 
$T^\pm M$ are integrable if and only if $P$ is a product structure. In this case, both $T^\pm M$ are
involutive, as we will assume. A simple computation shows that
$$
T_{p^\pm}=\frac{1}{4}T_P ,
$$
thus, in the case of a product structure we have $T_{p^-}=0$.

Of course, given a product structure on $M$ we can construct a Lie algebroid, induced by the operator
$D=\mathcal{L}_P$, applying proposition \ref{TN}. The bracket is then the contracted one,
$$
[\!|X,Y|\!]=[X,Y]_P.
$$
As we know, the Lie algebroid obtained in this way is isomorphic to the trivial one on $TM$
(recall proposition \ref{trivial}).
The results in section \ref{sec4} show that we can define another Lie algebroid structure.
\begin{thm}\label{PA}
Let $P\in\Omega^1(M;TM)$ be a product structure on the manifold $M$. There exists a Lie algebroid
structure on $M$, $(TM, q, [\![\, ,\,]\!])$, where the anchor map is 
$q=p^-=\frac{1}{2}(\mathrm{Id}-P)$,
and the bracket
$$
[\![X,Y]\!]=[X,Y]_{p^-}=\frac{1}{2}([X,Y]-[X,Y]_P).
$$
\end{thm}
\begin{proof}
The statement is actually a corollary to theorem \ref{Lct}. The last equality is just a trivial computation.
\end{proof}
\begin{rem}
The theorem applies also to the case of an operator $P\in\Omega^1(M;TM)$ such that 
$P^2=\epsilon^2\mathrm{Id}_{TM}$, with $\epsilon\in\mathbb{R}-\{0\}$.
\end{rem}

%If we have an operator $P\in\Omega^1(M;TM)$  such that $P^2=\epsilon^2Id_{TM}$, its eigenvalues are $\pm\epsilon$ and define the projection operators $p^{\pm}:=\frac{1}{2}\left(Id\pm\frac{1}{\epsilon}P\right)$. We have a decomposition $TM=T^{\epsilon+}M\oplusT^{\epsilon-}M $, where, $T^{\epsilon \pm}M=Imp^{\pm}$.\\
%We know that the distributions $T^{\epsilon \pm}M$ are integrable if and only if $P$ is integrable.
%In the case when $T_P=0$, because a simple calculation shows that $T_{p^{\pm}}=\frac{1}{4\epsilon^2}T_P$, we have that $T_p^-=0$.\\
%The results in section \ref{sec4} allow us to construct a Lie algebroid structure.
%
%\begin{thm}
%Let $P\in\Omega^1(M;TM)$ such that $P^2=\epsilon^2Id_{TM}$. There exists a Lie algebroid
%structure on $M$, $(TM, q, [\![\, ,\,]\!])$, where the anchor map is $q=\frac{1}{2}\left(Id-\frac{1}{\epsilon}P\right)$,
%and the bracket
%$$
%[\![X,Y]\!]=[X,Y]_{q}=\frac{1}{2}\left([X,Y]-\frac{1}{\epsilon}[X,Y]_P\right).
%$$
%\end{thm}
%The proof is the same as theorem \ref{PA}.

\section{The Lie algebroid associated to a tangent structure and a connection}\label{sec8}
Let us recall some facts about the geometry of the tangent bundle (see \cite{gri72}).
Let $\pi:TM\to M$ be the tangent bundle of a manifold $M$, and let $\tau:TTM\to TM$ be the 
second tangent bundle. Then, we have a short exact sequence of vector bundles
$$
0\longrightarrow TM\times_M TM \stackrel{\iota}{\longrightarrow}TTM\stackrel{j}{\longrightarrow}
TM\times_M TM\longrightarrow 0,
$$
where
$$
\iota (u,w)=\left. \frac{\mathrm{d}}{\mathrm{d}t}(u+tw)\right|_{t=0}
$$
is the natural injection, and $j=(\tau, \pi)$.

The vertical sub-bundle $\mathcal{V}TM=\ker\pi_*$ can be expressed in either form
$$
\mathrm{Im}\iota =\mathcal{V}TM=\ker j,
$$
and the vector-valued form $J\in\Omega^1(TM;TTM)$ defined by
$$
J=\iota \circ j,
$$
is called the vertical endomorphism. It is immediate to prove the following properties:
\begin{align*}
&J^2 =0 \\
&\ker J=\mathcal{V}TM =\mathrm{Im}J\\
&T_J=\frac{1}{2}[J,J]_{FN}=0.
\end{align*}
Thus, $J$ defines a (canonical) tangent structure on $TM$. It has locally constant rank
equal to $\dim M$.
We also have the (canonical) Liouville vector field on $TM$, $C\in\mathcal{X}(TM)$, defined as
$$
C=\iota \circ \Delta ,
$$
where $\Delta :TM\to TM\times_M TM$ is the diagonal $\Delta (v)=(v,v)$. The Liouville vector
field has the property 
$$
\mathcal{L}_C J=-J.
$$

A (nonlinear) connection $\Gamma$ on $TM$ is a vector-valued form $\Gamma\in\Omega^1(TM;TTM)$ such that
\begin{equation*}
J\circ \Gamma =J=-\Gamma\circ J.
\end{equation*}
A basic property of a connection is that it defines an almost product structure on $TM$, that is,
$$
\Gamma^2 =\mathrm{Id}_{TTM},
$$
in such a way that the eigenbundle corresponding to the eigenvalue
$\lambda =-1$ is precisely the vertical distribution. In other words, if we define the projectors
\begin{align*}
h&=\frac{1}{2}(\mathrm{Id}_{TTM}+\Gamma) \\
v&=\frac{1}{2}(\mathrm{Id}_{TTM}-\Gamma) ,
\end{align*}
we have $\mathcal{V}TM=\mathrm{Im}v$, and
$TTM=\mathrm{Im}v\oplus\mathrm{Im}h$.

Notice that the vertical distribution is integrable ($\mathrm{Im}v$ is involutive), while,
in general, the horizontal distribution is not (in fact, a necessary and sufficient 
condition for this is the
vanishing of the curvature of the connection).
Thus, we expect to be able to apply the results of section \ref{sec4}
in order to construct a Lie algebroid
structure where the vector bundle is $E=TTM$, and the anchor map $v$.
\begin{thm}
Let $M$ be a manifold, and $J$ the canonical tangent structure on $TM$. Let $\Gamma\in\Omega^1(TM;TTM)$
be a connection on $TM$. Then, there exists a Lie algebroid structure on $TM$, $(TTM,q,[\![\, ,\,]\!])$,
where the anchor map is $q=v=\frac{1}{2}(\mathrm{Id}_{TTM}-\Gamma)$ and the bracket
\begin{equation}\label{last}
[\![A,B]\!]=\frac{1}{2}\left( [A,B]-[A,B]_\Gamma \right)+\frac{1}{4}T_\Gamma (A,B).
\end{equation}
\end{thm}
\begin{proof}
This is just a corollary to theorem \ref{Lct}. Simply notice that, as a quick calculation shows,
$$
T_v(A,B)=\frac{1}{4}T_\Gamma (A,B),
$$
for all $A,B\in\mathcal{X}(TM)$, and 
\begin{align*}
[A,B]_v =& [vA,B]+[A,vB]-v[A,B] \\
		=& \frac{1}{2}([A-\Gamma A,B]+[A,B-\Gamma B]-[A,B]+\Gamma [A,B]) \\
		=& \frac{1}{2}([A,B]-[\Gamma A,B]-[A,\Gamma B]+\Gamma [A,B]) \\
		=& \frac{1}{2}([A,B]-[A,B]_\Gamma ).
\end{align*}
\end{proof}
The easiest method for obtaining connections on $TM$ is to choose a semispray.
A semispray (or a second-order differential equation, see \cite{APS60}) on $M$ is a vector field 
on $TM$, $S\in\mathcal{X}(TM)$, such that it is also a section of the second tangent bundle
$TTM\to TM$. Vector fields $S\in\mathcal{X}(TM)$ which are semisprays can be characterized
with the aid of the canonical structures on $TM$ as those satisfying
\begin{equation}\label{jsc}
J\circ S=C.
\end{equation}

It is well known (see \cite{APS60,leorod89}) that for any semispray $S$, 
the endomorphism $\Gamma=-\mathcal{L}_S J$ is a connection
on $TM$, with associated projectors
\begin{align*}
h&=\frac{1}{2}(\mathrm{Id}_{TTM}-\mathcal{L}_S J) \\
v&=\frac{1}{2}(\mathrm{Id}_{TTM}+\mathcal{L}_S J).
\end{align*}
\begin{rem}
This intimate relation between semisprays and connections has
been used in Physics to study such topics as the inverse problem
for Lagrangian dynamics \cite{cra81}, degenerate Lagrangian systems \cite{can86}, or
the geometry and mechanics of higher order tangent bundles \cite{amr91}.
\end{rem}
We have the following result.
\begin{thm}
Let $M$ be a manifold, and $J$ the canonical tangent structure on $TM$. Let $S\in\mathcal{X}(TM)$ be a semispray on 
$M$. Then, there exists a Lie algebroid structure on $TM$, $(TTM,q,[\![\, ,\,]\!])$, where the anchor map is
$q=v=\frac{1}{2}(\mathrm{Id}_{TTM}+\mathcal{L}_S J)$ and the bracket
\begin{equation}\label{vlast}
[\![A,B]\!]=\frac{1}{2}\left([A,B]-\underset{A,B,S}{\circlearrowright}[A,\mathcal{L}_S B]_J\right)+\frac{1}{4}T_{\mathcal{L}_S J}(A,B).
\end{equation}
\end{thm}
\begin{proof}
Rewrite \eqref{last} taking into account \eqref{jsc} so, for any $X\in\mathcal{X}(TM)$,
\begin{align*}
(\mathcal{L}_S J)X=& [S,JX]-J[S,X] \\
				  =& [S,X]_J-[JS,X] \\
				  =& [S,X]_J-[C,X].
\end{align*}
That is, $\Gamma X=[C,X]-[S,X]_J$. A long but otherwise straightforward computation leads to \eqref{vlast}.
\end{proof}

%For acknowledgements section, please don't number the section, please begin it with 

%*{Acknowledgements}
\section*{Acknowledgments} The authors express their gratitude to J. V. Beltr\'an and J. Monterde 
(both at the Universitat de Val\`encia, Spain), for many useful discussions.

% You may incorporate your references as follows in your main tex file.
% Using BibTex is not recommended but can be handled.

%\medskip
%% The data information below will be filled by AIMS editorial staff
%Received xxxx 20xx; revised xxxx 20xx.
%\medskip

\end{document}